\documentclass{amsart}

\usepackage{graphicx}
\usepackage{hyperref}
\usepackage{amsmath}
\usepackage{amssymb}
\usepackage{enumitem}

\usepackage[all]{xy}
\usepackage{pgf,tikz}
\usetikzlibrary{arrows}

\usepackage{xcolor}

\newtheorem{theorem}{Theorem}[section]

\newtheorem{corollary}[theorem]{Corollary}

\newtheorem{proposition}[theorem]{Proposition}
\theoremstyle{definition}

\theoremstyle{remark}
\newtheorem{remark}[theorem]{Remark}

\numberwithin{equation}{section}

\begin{document}

\title[Minimal number of singular fibers in Lefschetz fibrations over
  $T^2$]{On the minimal number of singular fibers in Lefschetz fibrations
  over the torus}

\author{Andr\'as I. Stipsicz} \address{MTA, R\'enyi institute of
  Mathematics, Budapest, Hungary} \email{stipsicz.andras@renyi.mta.hu}
\thanks{} 
\author{Ki-Heon Yun} \address{Department of mathematics,
  Sungshin Women's University, Seoul 02844, Korea}
\email{kyun@sungshin.ac.kr}%
\thanks{We would like to thank the referee for many valuable comments and F. Catanese for pointing out an error in an earlier version of this paper. AS was partially supported by ERC Advanced Grant LDTBud and
by the \emph{Momentum} program of the Hungarian Academy of Sciences.
KY was partially supported
by the Basic Science Research Program through the National Research
Foundation of Korea (NRF) funded by the Ministry of Education
(2015R1D1A1A01058941) and by Sungshin Women's University study abroad
program.  }

\date{\today}

\begin{abstract}
We show that the minimal number of singular fibers $N(g,1)$ in a
genus-$g$ Lefschetz fibration over the torus is at least $3$. As an
application, we show that $N(g, 1) \in \{ 3, 4\}$ for $g\ge 5$, $N(g, 1) \in \{3,
4,5 \}$ for $g= 3, 4$ and $N(2,1) = 7$.
\end{abstract}

\maketitle

\section{introduction}
Due to results of Donaldson and Gompf, the study of symplectic
$4$-manifolds is closely related to the study of Lefschetz pencils and
Lefschetz fibrations.  A central open problem in four-manifold theory
is the geography problem of symplectic 4-manifolds and of Lefschetz
fibrations: if $\chi $ and $\sigma$ denote the Euler characteristic
and the signature of a 4-manifold, what relations do these integers
satisfy and which pairs can be realized by (minimal) symplectic
$4$-manifolds or by Lefschetz fibrations.  In particular, it is
expected that for any symplectic 4-manifold $X$ not diffeomorphic to
the blow-up of a ruled surface we have $\chi (X)\geq 0$ (\cite[page
  579]{Gompf:1995} and \cite[Conjecture~2.10]{Stipsicz:TAIA-2002}),
and for a Lefschetz fibration over the sphere we
have $\sigma \leq 0$ --- cf.
\cite[Problems~7.4~and~7.5]{Korkmaz-Stipsicz:IRMA-2009}.

Even though there is significant
progress in signature computations \cite{Endo:MathAnn-2000,
  Ozbagci:PJM-2002, Smith:GT-1999}, and the signature of a Lefschetz
fibration is algorithmically computable, its global properties are not
known except for hyperelliptic Lefschetz fibrations.  

The minimal
number of singular fibers of Lefschetz fibrations seems to be closely
related to these questions; this problem has been extensively studied
in \cite{Baykur-Korkmaz:arxiv-2015, Hamada:MMJ-2014,
  Korkmaz-Ozbagci:PAMS-2001, Korkmaz-Stipsicz:IRMA-2009,
  Monden:TokyoJ-2012, Stipsicz:TAIA-2002}.  (In what follows, we will
always assume that the fibration map is injective on the set of
critical points, hence the number of singular points and the number of
singular fibers coincide. To avoid trivial examples, we will also
assume that all Lefschetz fibrations are relatively minimal and admit
at least one singular fiber.)  
Suppose that $\Sigma _h$ denotes the
closed, oriented surface of genus $h$.  Let the positive integer
$N(g,h)$ be defined as
\begin{align*}
N&(g, h) =  \\ & \text{min} \{ k \ | \ 
\text{there is a genus-}g \text{ Lefschetz fibration over } \Sigma_h 
\text{ with } k>0\text{ singular fibers} \}. 
\end{align*}
The value of $N(g, h)$ has already been computed except $N(g, 0)$ for
$g\ge 3$, $N(g,1)$ for $g\ge 2$ and $N(2,2)$. Recently Baykur and
Korkmaz~\cite{Baykur-Korkmaz:arxiv-2015} found an interesting relation
in the mapping class group $\mathcal{M}_2^1$ of the genus-$2$ surface 
with one boundary component, and by using this relation they showed
that $N(2,0)= 7$. Furthermore, using the $8$-holed torus relation~\cite{Korkmaz-Ozbagci:MMJ-2008}
and the Matsumoto-Cadavid-Korkmaz relation
~\cite{Cadavid:Thesis-1998, Korkmaz:IMRN-2001, Matsumoto:1996},
Hamada~\cite{Hamada:MMJ-2014}  found an upper
bound for $N(g,1)$: he showed that $N(g,1)\leq 4$ if $g\geq 5$ and 
$N(g,1)\leq 6$ if $g=3,4$.

In this paper we estimate $N(g,1)$ further by using some
constraints on the signature of a Lefschetz fibration over the torus
$T^2$.

\begin{theorem}\label{Theorem:LowerBound}
Let $f\colon X \to T^2$ be a genus-$g$ Lefschetz
fibration. Then $f$ has at least three
singular fibers, that is, $N(g, 1) \ge 3$ for $g \ge 1$.
\end{theorem}

By lifting the Baykur-Korkmaz relation to
$\mathcal{M}_2^2$ we also show 
\begin{theorem}\label{Theorem:UpperBound}
If $g\geq 3$, then there is a genus-$g$
Lefschetz fibration over the torus with $5$ singular
fibers, that is, $N(g, 1) \le 5$ for $g \ge 3$.
\end{theorem}

When combined these two theorems with sharper results of
Hamada~\cite{Hamada:MMJ-2014} and some special properties of 
the mapping class group of the genus-$2$ surface, we get:

\begin{corollary}\label{Corollary:N}
For the minimal number $N(g,1)$ of singular fibers in a genus-$g$
Lefschetz fibration over the torus we have
\begin{itemize}
\item[(1)] $3 \le N(g, 1) \le 4$ for $g \ge 5$,
\item[(2)] $3\le N(g, 1) \le 5$ for $g = 3, 4$, 
\item[(3)] $N(2,1) = 7$ and $N(1,1) =12$.
\end{itemize}
\end{corollary}

\begin{remark}\label{Remark:Catanese}
In ~\cite{Cartwright-Steger:2010}, the existence of a complex surface $S$ with $b_1(S) = 2$, $\sigma(S) =1$ and $\chi(S) =3$ is shown.
The Albanese map $S\to T^2$ is studied in~\cite{Cartwright:2014aa}, where it has been shown that this map can be perturbed to a symplectic Lefschetz fibration of genus $19$ with exactly $3$ singular fibers. When combined this construction with the result above, we get $N(19, 1)= 3$.  The exact value of $N(g,1)$ for other values of $g$ is still open.
\end{remark}
\section{A lower bound on $N(g,1)$}

We start our discussion by recalling some relevant results. 
(For basic notions and constructions of Lefschetz fibrations see
\cite[Chapter~8]{GS}. Unless otherwise stated, 
in the following we will consider
only relatively minimal Lefschetz fibrations with at least one
singular fiber, and will assume that the fibration map is injective on
the set of critical points.)

\begin{theorem}[\cite{Li:IMRN-2000, Stipsicz:PAMS-2000, Stipsicz:PAMS-2000-1}] \label{Theorem:Li}
Let $X$ be a connected smooth closed oriented $4$-manifold and $f\colon X
\to B$ be a Lefschetz fibration with
fiber $F$. If $g(F) \ge 2$ and $g(B) \ge 1$, then
\begin{equation}\label{eq:ChernNumbers}
2(g(F) -1)(g(B)-1) \le c_1^2(X) \le 5 c_2(X) . 
\end{equation}
In particular, a Lefschetz fibration over the torus with fiber genus
at least 2 satisfies $0\leq c^2_1(X)\leq 5c_2(X)$. \qed
\end{theorem}

\begin{theorem}[\cite{Braungardt-Kotschick:TAMS-2003, Endo-Kotschick:InventMath-2001}]\label{Theorem:Braungardt-Kotschick}
Let $X$ be a connected smooth closed oriented $4$-manifold and
$f\colon X \to B$ be a Lefschetz fibration with
fiber $F$ where $g(F) \ge 1$ and $g(B) \ge 1$. Let $n$ be the number
of nonseparating and $s$ be the number of separating vanishing
cycles. Then,
\[
s \le 6(3g(F) -1)(g(B)-1) + 5n.
\]
In particular, for a Lefschetz fibration over the torus we have $s\leq 5n$. \qed
\end{theorem} 
 
The following proposition will be used to show that $N(g,1)\neq 2$.

\begin{proposition}\label{Proposition:SigmaEstimate}
Let $f\colon X \to T^2$ be a genus-$g$ Lefschetz fibration over the
torus $T^2$ which has $k$ singular fibers, $n$ of them with
nonseparating vanishing cycles and $s$ of them with separating
vanishing cycles.  Then
\[
-\frac{2}{3}k \le \sigma(X) \le n-s-1.
\]
\end{proposition}

\begin{proof}
Note first that for a Lefschetz fibration $f\colon X\to B$
of fiber genus $g$ and base genus
$h$ and with $k$ singular fibers we have $\chi (X)=4(g-1)(h-1)+k$;
in particular if $B$ is diffeomorphic to the torus, then
$\chi (X)=k$.

The first inequality of the proposition now follows from
Equation~\eqref{eq:ChernNumbers}, as
\[
0\leq c_1^2(X) =3\sigma(X) + 2\chi(X) = 3\sigma (X) +2k. 
\]

For the second inequality, let us consider an $m$-fold (unbranched) covering 
\[
\phi_m \colon T^2 \to T^2
\]
of the base torus $T^2$ and pull back the Lefschetz fibration $f\colon
X \to T^2$  along this map: 
\[
\xymatrix{
\phi_m^*(X) \ar[r]\ar[d]^{f_m} & X \ar[d]^{f} \\
T^2 \ar[r]^{\phi_m} & T^2
}
\]

The Euler characteristic, the signature 
and the number of singular fibers all multiply by $m$, hence we get:
\begin{itemize}
\item[(1)] $\chi(\phi_m^*(X))  =  m k = m n + m s$
\item[(2)] $b_1(\phi_m^*(X)) \le 2 g + 2$
\item[(3)] $\chi(\phi_m^*(X)) = 2 -2b_1(\phi_m^*(X)) +
  b_2(\phi_m^*(X))=m\chi (X) =mk$
\item[(4)] $\sigma(\phi_m^*(X)) = m \sigma(X)$.
\end{itemize}
Therefore 
\begin{align*}
&\sigma(\phi_m^*(X)) = b_2^+(\phi_m^*(X)) - b_2^-(\phi_m^*(X)) = b_2(\phi_m^*(X)) - 2 b_2^-(\phi_m^*(X))= \label{Equation:Signature}\\
&= \chi(\phi_m^*(X)) - 2 + 2 b_1(\phi_m^*(X))  - 2b_2^-(\phi_m^*(X)) \le mk + 4g + 2 -2b_2^-(\phi_m^*(X)), \notag
\end{align*}
and we will estimate $b_2^-(\phi_m^*(X))$.

In the mapping class group of the fiber,
let us denote $yxy^{-1}$ by $[x]^y$. We will consider $a_i$, $b_i$ as
a right handed Dehn twists along corresponding simple closed curves on
the fiber $F$. In the following we use the function notation for a product
in the mapping class group, \emph{i.e.}, $ab = a\circ b$ means that $b$
is applied first and then we apply $a$.  For simplicity we will use
the same letter for a Dehn twist and its corresponding isotopy class
of simple closed curve.

Since the 4-manifold $X$ admits a Lefschetz fibration over $T^2$, it
is characterized by the equivalence class of the monodromy
factorization of the fibration of the form:
\[
a_1 a_2 \cdots a_s b_1 b_2 \cdots b_n = [\gamma, \delta]
\]
for some right handed Dehn twists $a_i$ $(i=1, \cdots, s)$ along
separating vanishing cycles and right handed Dehn twists $b_j$ ($j=1,
2, \cdots, n$) along nonseparating vanishing cycles, and some elements
$\gamma$, $\delta$ in the mapping class group of the fiber $F$.

The monodromy factorization of the Lefschetz fibration $\phi_m^*(X)$ has
the form 
\[
\tilde{a}_1 \tilde{a}_2 \cdots \tilde{a}_{ms} \tilde{b}_1 \tilde{b}_2 \cdots \tilde{b}_{mn} = [\gamma_m, \delta_m]
\]
for some right handed Dehn twists  $\tilde{a}_i$ ($i=1, \cdots, ms$) along separating vanishing cycles and right handed Dehn twists $\tilde{b}_j$ ($j=1, \cdots, mn$)  along  nonseparating vanishing cycles and some elements $\gamma_m$, $\delta_m$ in the mapping class group of the fiber $F$.

Each separating vanishing cycle gives rise to an embedded surface
(hence a second homology element $A_i$) of self-intersection $(-1)$
which is part of the corresponding singular fiber. The Lefschetz
fibration $\phi_m^*(X)$ therefore has $ms$ many such disjoint surfaces
\[
\{ A_1, A_2, \cdots, A_{ms} \}.
\]

Now we claim that there is a surface of negative self intersection in a neighbourhood of singular fibers $\prod_{i=1}^{2g+1} \tilde{b}_{(2g+1)j + i}$ for each $j=0, \cdots, [\frac{mn}{2g+1}] -1$. 

Since $\dim H_1(F; \mathbb{Z}) = 2g$, 
\[
\{ \tilde{b}_{(2g + 1) j + 1},  \tilde{b}_{(2g + 1) j + 2}, \cdots,  \tilde{b}_{(2g + 1) (j + 1)}\}
\]
is linearly dependent for each $j=0, \cdots, [\frac{mn}{2g+1}]-1$. 
For each $j$ all argument will be the same, hence to keep the notation simple, we assume $j=0$.
Obviously we can find a positive integer $\ell \le 2g +1$ such that 
\[
\{ \tilde{b}_1, \tilde{b}_2, \cdots, \tilde{b}_{\ell -1}\}
\]
is linearly independent and 
\[
\{ \tilde{b}_1, \tilde{b}_2, \cdots, \tilde{b}_{\ell}\}
\]
is linearly dependent. Then we can find some integers $n_v$ such that
\[
\sum_{v=1}^{\ell-1} n_v \tilde{b}_v +  n_\ell \tilde{b}_\ell = 0, \ n_\ell \ne 0
\]
in $H_1(F; \mathbb{Z})$.
Since this implies the same identity in
$H_1(\partial (F\times D^2); \mathbb{Z})$, there is an element $\alpha
\in H_2(F\times D^2, \partial(F\times D^2); \mathbb{Z})$ which can be
represented by a surface $S_{\alpha}$ in $F\times D^2$ with $\sum_{v=1}^{\ell} \vert n_v \vert$
boundary components, $\vert n_v \vert$ copies of $\tilde{b}_v$ for each $v=1, \cdots, \ell$, all located in different fibers. 
Using the Lefschetz thimbles of
the singular fibers in $X$ corresponding to $\tilde{b}_1, \tilde{b}_2, \cdots, \tilde{b}_\ell$ we
get a closed, oriented surface $S$ in $X$.

Next we show that the self-intersection of $S$ is negative. 
Decompose $S$ as $S_{\alpha}$ together with the $\sum_{v=1}^{\ell} \vert n_v \vert$ Lefschetz thimbles.
When framing the $\sum_{v=1}^{\ell} \vert n_v \vert$ boundary circles of $S_{\alpha}$ with the framing 
these circles get from the fiber of the fibration, the self-intersection
$[S]^2$ can be computed as the sum of $[S_{\alpha}]^2$ (with respect to
the above framing at the boundary) together with the self-intersections
of the thimbles (again, with the same framings at their boundaries).
The relative self-intersection of the Lefschetz
thimble is equal to $(-1)$, hence the contribution of the $\sum_{v=1}^{\ell} \vert n_v \vert$ thimbles
in our case is $-\sum_{v=1}^{\ell} n_v^2 $ (since $\tilde{b}_v$ appears $\vert n_v\vert$ times, hence this
thimble is used $\vert n_v\vert$ times).

In computing the self-intersection of the surface-with-boundary
$S_{\alpha}\subset F\times D^2$ (with the framings along the $\sum_{v=1}^{\ell} \vert n_v \vert$
boundary components as fixed above) we argue as follows.  We attach
$\ell$ two-handles to $F\times D^2$ along the simple closed curves $\tilde{b}_1$,
$\tilde{b}_2$, $\cdots$, $\tilde{b}_\ell$, with coefficient $0$ (measured with respect
to the framing given above).  The resulting 4-manifold $\bar{X}$
satisfies $H_2(\bar{X};\mathbb{Z}) = \mathbb{Z}\oplus \mathbb{Z}$: 
the first $(\ell-1)$ handles are attached along homologically nontrivial and
linearly independent curves, hence reduce $H_1$ and do not change
$H_2$, while the last 2-handle is attached along a curve
homologically depending on the first $(\ell-1)$ curves, hence this last
attachment increases the rank of $H_2$ by one.  
The two generators of
$H_2(\bar{X};\mathbb{Z})$ can be given by $F\times \{ 0\}$ and the
capped-off surface we get from $S_{\alpha}$, and which we will denote
by $\Sigma$.  These classes can be represented by disjoint surfaces,
since $S_{\alpha}$ can be pushed into $\partial (F\times D^2)$;
consequently $[F\times \{ 0\}]\cdot [\Sigma ]=0$.  The
self-intersection of the homology class represented by $F\times \{
0\}$ is also zero, hence the sign of the self-intersection of $\Sigma$
is equal to the sign of the signature of $\bar{X}$.  
The signature
$\sigma(\bar{X})$, however, is easy to compute: the first $(\ell-1)$
attachments do not change $\sigma (F\times D^2)=0$ (since they do not
change $H_2(F\times D^2; \mathbb{Z} )$ either), and by Wall's nonadditivity
of signature (\cite[Theorem~3]{Ozbagci:PJM-2002}) the last handle
attachment also leaves $\sigma$ unchanged, since $\vert n_\ell \vert$ copies of the simple closed
curve corresponding to the $0$-framing in the $2$-handle attachment
along $\tilde{b}_\ell$ bounds a surface (cf. the formula in
\cite{Ozbagci:PJM-2002}).)  This shows that the self-intersection of
$\Sigma$ is zero in $\bar{X}$, hence the self-intersection of
$S_{\alpha}$ in $F\times D^2$ (with the framings along the boundary
components as discussed above) is $0$, implying that $S\subset X$ is a
surface of self-intersection $-\sum_{v=1}^{\ell} n_v^2 $. 

The surface $S$ is obviously disjoint from the $A_i$'s (coming from separating vanishing cycles), 
and the $S$'s for different $j$'s are also disjoint.
Therefore we get
\[
b_2^-(\phi_m^*(X)) \ge ms + [\frac{mn}{2g+1}],
\]
implying
\[
m \sigma(X) \le m k -2  + 2(2g + 2) - 2(ms + [\frac{mn}{2g+1}])
\]
for each positive integer $m$. This implies
\[
\sigma(X) \le k -2s -\frac{2n}{2g+1},
\]
and since $\sigma(X)$ is an integer, together with the fact that $n>0$ (Theorem~\ref{Theorem:Braungardt-Kotschick}) we get
\[
\sigma(X) \le k -2s -1.
\]
\end{proof}

\begin{proof}[Proof of Theorem~\ref{Theorem:LowerBound}]
Suppose that there is a genus-$g$ Lefschetz fibration $f\colon X \to T^2$ with
$k=n+s$ singular fibers. Then, the
following conditions have to be satisfied:
\begin{align}
& -\frac{2}{3}k \le \sigma(X) \le n-s-1 \text{ (Proposition~\ref{Proposition:SigmaEstimate})},
 \label{Equation:C1}\\
& 4 | (\sigma(X) + \chi(X) ) \text{ (by the existence of 
an almost complex structure on $X$)},
\label{Equation:C2} \\
& s \le 5n \text{ (Theorem~\ref{Theorem:Braungardt-Kotschick})}. \label{Equation:C3}
\end{align}

 It is known that $N(1,1) = 12$
\cite{Matsumoto:1986}, $N(g, 1) >1$ for $g \ge 1$
\cite{Korkmaz-Ozbagci:PAMS-2001}, and $N(2,1) \ge 6$ because
\eqref{Equation:C3} implies $n>0$ \cite{Matsumoto:1996,
  Monden:TokyoJ-2012}.  Notice that $N(g,1)>1$ also follows from
\eqref{Equation:C1}: if $k=1$ then \eqref{Equation:C1} implies
$-\frac{2}{3} \le \sigma(X) \le 0$, so $\sigma(X) =0$, contradicting \eqref{Equation:C2}.

For $k=2$, we also use \eqref{Equation:C2}: there is no integer $\sigma (X)$
satisfying
\[
-\frac{4}{3} \le \sigma(X) \le 2 - 2s -1= 1-2s
\]
and $4 \vert (\sigma(X)+2)$.
\end{proof}

\begin{remark}
If there is a Lefschetz fibration $X$ over $T^2$ with
$k = 3$, then by \eqref{Equation:C1} and \eqref{Equation:C2} we have $\sigma(X) = 1$, $\chi(X) = 3$ and  $(n,s)=(3,0)$. 
As already pointed out in Remark~\ref{Remark:Catanese}, such a fibration has been claimed in ~\cite{Cartwright:2014aa} for fiber genus $19$, based on the construction of a complex surface $S$ in ~\cite{Cartwright-Steger:2010}.
\end{remark}

\section{Upper bounds on $N(g,1)$}

Hamada proved the following upper bounds for $N(g, 1)$, $g \ge 3$.

\begin{theorem}[\cite{Hamada:MMJ-2014}]\label{Theorem:Hamada}
For the minimal number $N(g,1)$ of singular fibers in a genus-$g$
Lefschetz fibration over the torus we have
\begin{itemize}
\item[(1)] $N(g, 1) \le 6$ for $g \ge 3$,
\item[(2)] $N(g, 1) \le 4$ for $g \ge 5$. \qed
\end{itemize}
\end{theorem}
The proof of this result uses the Matsumoto-Cadavid-Korkmaz word 
\cite{Cadavid:Thesis-1998, Korkmaz:IMRN-2001, Matsumoto:1996}, implying
$N(g, 1) \le 6$ for $g\geq 3$,  and the $8$-holed torus relation (originally
obtained by Korkmaz and Ozbagci \cite{Korkmaz-Ozbagci:MMJ-2008}), implying
$N(g, 1) \le 4$ for $g\geq 5$.

By modifying the $4$-chain relation, Baykur and Korkmaz
\cite{Baykur-Korkmaz:arxiv-2015} found an interesting relation in
$\mathcal{M}_2^1$.  In the following we lift the Baykur-Korkmaz
relation (a relation which is obtained from
Equation~\eqref{Equation:Baykur-Korkmaz} below by capping off the boundary
circle $\delta_0$ with a disk) to $\mathcal{M}_2^2$, and using this
lift we improve the upper bounds for $N(3,1)$ and $N(4,1)$. This
result was predicted by Hamada and a sketch is given in
\cite{Hamada:MMJ-2014}.

\begin{proposition}\label{Proposition:Lifting-1}
There is a relation
\begin{equation}\label{Equation:Baykur-Korkmaz}
x_1 x_2 x_3 x_4 y_1 y_2 y_3 = \delta_0 \delta_2
\end{equation}
in the mapping class group $\mathcal{M}_2^2$ of the genus-$2$ surface with
two boundary components $\delta_0$ and $\delta_2$, where $x_i$ are
nonseparating vanishing cycles and $y_i$ are separating vanishing
cycles.
\end{proposition}

\begin{proof}
\begin{figure}[htb]
\begin{tikzpicture}[scale=0.65]

\draw [blue] (0,4) to [out=240, in =120] (0,0);
\draw [blue] (0,4) to [out=300, in =60] (0,0)  node [above] at (0,4) {$\delta_0$};
\draw[dashed, blue] (10,4) to [out=240, in =120] (10,0) node [above] at (10,4) {$\delta_2$};
\draw [blue] (10,4) to [out=300, in =60] (10,0);
\draw [dashed, blue] (5,4) to [out=240, in =120] (5,0) node [above] at (5,4) {$\delta_1$};
\draw [blue] (5,4) to [out=300, in =60] (5,0);
\draw [thick] (3,2) circle [radius=0.5];
\draw [thick] (7,2) circle [radius=0.5];
\draw [blue] (3,2) circle [radius=1] node [left] at (2,2) {$2$};
\draw [blue] (7,2) circle [radius=1] node [left] at (8,2) {$4$};
\draw [blue] (3.5,2) to [out=30, in =150] (6.5,2) node [below] at (5,2.5) {$3$};
\draw[dashed, blue] (3.5,2) to [out=330, in =210] (6.5,2);
\draw[dashed, blue] (3,4) to [out=240, in =120] (3,2.5) node [above] at (3.3,4) {$1_a$};
\draw [blue] (3,4) to [out=300, in =60] (3,2.5);
\draw[dashed, blue] (3,1.5) to [out=240, in =120] (3,0);
\draw [blue] (3,1.5) to [out=300, in =60] (3,0) node [below] at (3.3,0) {$1_b$};
\draw[dashed, blue] (7,4) to [out=240, in =120] (7,2.5) node [above] at (7,4) {$c$};
\draw [blue] (7,4) to [out=300, in =60] (7,2.5);
\draw[dashed, blue] (7,1.5) to [out=240, in =120] (7,0);
\draw [blue] (7,1.5) to [out=300, in =60] (7,0) node [below] at (7,0) {$d$};
\draw [thick, red] (0,4)--(10,4) node [above] at (1,4) {$\alpha$};
\draw [thick, red] (0,0)--(10,0) node [below] at (1,0) {$\beta$};

\draw[dashed] (3.25,2.4) to [out=30, in =180] (7,3.5);
\draw[dashed] (7,3.5) to [out=0, in =90] (8.5,2);
\draw[dashed] (3.25,1.6) to [out=330, in =180] (7,0.5);
\draw[dashed] (7,0.5) to [out=0, in =270] (8.5,2);

\draw (3.25,2.4) to [out=40, in =180] (7,3.7);
\draw (7,3.7) to [out=0, in =90] (8.7,2);
\draw (3.25,1.6) to [out=320, in =180] (7,0.3);
\draw (7,0.3) to [out=0, in =270] (8.7,2) node [right] at (8.7,2) {$\delta_3$};;

\end{tikzpicture}
\caption{Curves for the chain relation}\label{Figure:Circles}
\end{figure}
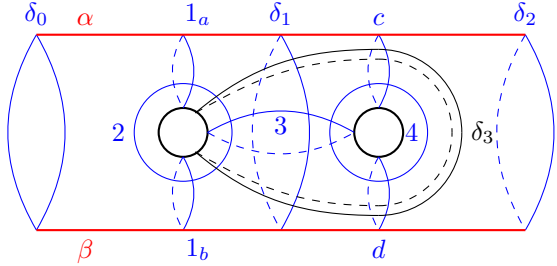

By using a sequence of braid relations, we get
\begin{align*}
k 1_a 2 3 4 &= 1_a 2 3 4 \ell, &  k 1_b 2 3 4 &= 1_b 2 3 4 \ell, & 2 1_a 2 3 4 &= 1_a 2 3 4  1_a, & 2 1_b 2 3 4 &= 1_b 2 3 4 1_b,   \\
\ell 4 3 2 1_a &= 4 3 2 1_a k, & \ell 4 3 2 1_b &= 4 3 2 1_b k, & 1_a 4 3 2 1_a &= 4 3 2 1_a 2, & 1_b 4 3 2 1_b &= 4 3 2 1_b 2 , 
\end{align*}
where $k \in \{3, 4\}$ and $k-\ell = 1$. 
This implies that the elements of the set  
$\{ (1_a 1_b), (2 1_a 1_b 2), ( 3 2 1_a 1_b 2 3), ( 4 3 2 1_a 1_b 2 3 4)\}$ commute with each other and
\begin{align*}
(1_a 2 3 4 1_b 2 3 4)^{2} (1_a 2 3 4) &=  1_a 2 3 1_b 2 3 1_a 2 3 1_b 2 3 4 3 2 1_b 1_a 2 3 4 \\
& = 1_a 2  1_b 2  1_a 2 3 2 1_a 1_b 2 3 4 3 2 1_b 1_a 2 3 4\\
&= (1_a  1_b) (2 1_a 1_b 2)( 3 2 1_a 1_b 2 3)( 4 3 2 1_a 1_b 2 3 4),\\
(1_b 2 3 4)(1_a 2 3 4 1_b 2 3 4)^{2} &= (1_a  1_b) (2 1_a 1_b 2)( 3 2 1_a 1_b 2 3)( 4 3 2 1_a 1_b 2 3 4).
\end{align*}
Therefore 
\[
(1_a 2 3 4 1_b 2 3 4)^{5} = (1_a  1_b)^2  (2 1_a 1_b 2)^2 ( 3 2 1_a 1_b 2 3)^2 ( 4 3 2 1_a 1_b 2 3 4)^2.
\]

First we prove the following relation in $\mathcal{M}_2^2$:
\begin{equation}\label{Equation:Lift}
(1_a 2 3 4 1_b 2 3 4)^{5} = \delta_0^3 \delta_2.
\end{equation}
(Note that if we attach a disk along $\delta_0$, then the above
relation reduces to the usual $4$-chain relation $(1234)^{10} =
\delta_2$, hence we can regard it as a lift of the $4$-chain relation to
$\mathcal{M}_2^2$.)  The $3$-chain relation implies
\begin{equation}\label{Equation:3-Chain}
 (1_a 1_b)^2 (2 1_a 1_b 2)^2 = \big( (1_a 1_b) (2 1_a 1_b 2) \big)^2 = (1_a 2 1_b)^4 = \delta_0 \delta_1.
\end{equation}
Since the genus-$2$ surface $\Sigma_2^2$ with two boundary components can be decomposed into eight hexagons by cutting it along seven simple closed curves $1_a$, $1_b$, $2$, $3$, $4$, $c$, $d$ and two arcs $\alpha$, $\beta$ as in Figure~\ref{Figure:Circles}, 
we will prove 
\begin{equation}\label{Equation:Rel-2}
( 3 2 1_a 1_b 2 3)^2 ( 4 3 2 1_a 1_b 2 3 4)^2 = \delta_0^2 \delta_1^{-1} \delta_2
\end{equation}
by showing that nine circles $1_a$, $1_b$, $2$, $3$, $4$, $c$, $d$, $\delta_0$, $\delta_2$ and two arcs $\alpha$, $\beta$ are fixed (up to isotopy) under the map 
\begin{equation}\label{Equation:Rel}
(\delta_0^2 \delta_1^{-1} \delta_2)^{-1} ( 3 2 1_a 1_b 2 3)^2 ( 4 3 2 1_a 1_b 2 3 4)^2 .
\end{equation}
\begin{itemize}
\item $\delta_0$ and $\delta_2$ are clearly fixed under the map \eqref{Equation:Rel}.
\item $c$ and $d$ are fixed under $( 4 3 2 1_a 1_b 2 3 4)^2$ and  $(3 2 1_a 1_b 2 3)$, 
  hence are fixed under the map \eqref{Equation:Rel}.
\item $(3 2 1_a 1_b 2 3)(4 3 2 1_a 1_b 2 3 4)$ maps $1_a$, $1_b$ and $2$ to themselves with the same orientation. So $1_a$, $1_b$ and $2$ are fixed under the map \eqref{Equation:Rel}.  
\item $(3 2 1_a 1_b 2 3) (4 3 2 1_a 1_b 2 3 4)$ maps $4$ to itself but with opposite orientation. 
  So $4$ is fixed under the map $(3 2 1_a 1_b 2 3)^2 (4 3 2 1_a 1_b 2 3 4)^2$ and therefore by \eqref{Equation:Rel}.  
\item $ (4 3 2 1_a 1_b 2 3 4)$ maps the simple closed curve $3$ to itself
  but with opposite orientation. The image of the simple closed curve $3$
  under the map $(3 2 1_b 1_a 2 3)^2$ is  $\delta_1^{-1}(3)$.
  Therefore $3$ is fixed under the map \eqref{Equation:Rel}.
\item $(3 2 1_a 1_b 2 3)^2 (4 3 2 1_a 1_b 2 3 4)^2$ maps $\alpha$ to $\delta_0^2 \delta_1^{-1} \delta_2(\alpha)$ and $\beta$ to $\delta_0^2 \delta_1^{-1} \delta_2(\beta)$. So $\alpha$ and $\beta$ are fixed under the map \eqref{Equation:Rel}.
\end{itemize}
Equations \eqref{Equation:3-Chain} and \eqref{Equation:Rel-2} then
imply Equation \eqref{Equation:Lift}. 

Now we will check that there is a lift of the Baykur-Korkmaz word to
$\mathcal{M}_2^2$.  
By following the steps of the
reverse engineering method applied by Baykur and Korkmaz in
\cite{Baykur-Korkmaz:arxiv-2015}, we get
\begin{align*}
& (3 2 1_a 1_b 2 3)^2 ( 4 3 2 1_a 1_b 2 3 4)^2 = (3 2 1_a 1_b 2 3) (4 3 2 1_a 1_b 2 3 4)^2 (3 2 1_a 1_b 2 3) =\\
&= 3 4 2 1_a 1_b 3 2 3 1_a 1_b 2 (3 4)^3 2 1_a 1_b 2 3 2 1_a 1_b 2 4 3.
\end{align*}
This implies
\begin{align*}
\delta_0^3 \delta_2 &= \delta_0 \delta_1 3 4 2 1_a 1_b 3 2 3 1_a 1_b 2 (3 4)^3 2 1_a 1_b 2 3 2 1_a 1_b 2 4 3,
\end{align*}
and since $4 3$ commutes with $\delta_0^3 \delta_2$, we can perform a
cyclic permutation; a sequence of Hurwitz moves then gives
\begin{align*}
\delta_0^3 \delta_2 &= \delta_0 [\delta_1]^{43}  4 3 3 4 2 1_a 1_b 3 2 3 1_a 1_b 2 (3 4)^3 2 1_a 1_b 2 3 2 1_a 1_b 2 \\
&=  \delta_0 [\delta_1]^{43} [2]^{4334} [2]^{43341_a1_b3} [2]^{43341_a1_b 331_a1_b} (1_a1_b)^2 (2 1_a1_b2)^2 (34)^6 [3]^{(21_a1_b2)^{-1}} \\
&= \delta_0 [\delta_1]^{43} [2]^{4334} [2]^{43341_a1_b3} [2]^{43341_a1_b 331_a1_b} \delta_0 \delta_1 \delta_3 [3]^{(21_a1_b2)^{-1}}
\end{align*}
where $\delta_3$ is the right handed Dehn twist along the curve $\delta_3$ of Figure~\ref{Figure:Circles}.
Therefore
\[
\delta_0 \delta_2 = [3]^{(21_a1_b2)^{-1}} [\delta_1]^{43} [2]^{4334} [2]^{43341_a1_b3} [2]^{43341_a1_b 331_a1_b} \delta_1 \delta_3 .
\]

By choosing
\begin{align*}
x_1 &= [3]^{(21_a1_b2)^{-1}}, &   x_2  &= [2]^{4334}, &
x_3 &= [2]^{43341_a1_b3},  & x_4 &= [2]^{43341_a1_b 331_a1_b}, \\ 
y_1 &=  [\delta_1]^{(x_2x_3x_4)^{-1}43},  & y_2 &= \delta_1, & 
y_3 &= \delta_3, & &
\end{align*}
the conclusion follows.
\end{proof}

\begin{remark}
Notice that Relation~\eqref{Equation:Lift} shows the existence of a $(-3)$-section of the Lefschetz fibration
determined by the left-hand-side of the relation. This fibration was already discussed by Fuller
\cite[page~164]{Fuller:1998}.
\end{remark}

\begin{proof}[Proof of Theorem~\ref{Theorem:UpperBound}]
Let us cap off the two boundary components $\delta_0$, $\delta_2$ of
$\Sigma_2^2$ by a cylinder so that we get a closed surface $\Sigma_3$
of genus $3$. Then each vanishing cycle except $y_3$ is nonseparating
and $\Sigma_3 - \{ \delta_0, x_2\}$ and $\Sigma_3 -\{ \delta_2, x_1\}$
are connected, $\delta_0$ is disjoint from $x_2$, and $\delta_2$ is
disjoint from $x_1$.  As it is explained in \cite{EKKOS:Topology-2002,
  Korkmaz-Ozbagci:PAMS-2001}, there is a map $\psi \colon \Sigma_3 \to
\Sigma_3$ satisfying
\[
\psi(\delta_0) = x_1, \ \ \psi(x_2) = \delta_2,
\]
implying that
\[
x_3 x_4y_1 y_2 y_3  = [\delta_0 x_2^{-1}, \psi].
\]
So there is a genus-$3$ Lefschetz fibration over $T^2$ with $5$
singular fibers, $4$ of them are nonseparating and $1$ of them is
separating.

We get the same result for $g \ge 4$, since we can cap off the two
boundary components $\delta_0$ and $\delta_2$ by a twice punctured
genus-$(g-3)$ surface.
\end{proof}

As an application, we get

\begin{proof}[Proof of Corollary~\ref{Corollary:N}]
Theorems~\ref{Theorem:LowerBound}, \ref{Theorem:Hamada}
and~\ref{Theorem:UpperBound} immediately imply (1) and (2).

For (3), it is already known that $N(1,1)=12$ and $6\le N(2,1) \le 7$. So we only need
to show that there is no genus-$2$ Lefschetz fibration $X$ over $T^2$
with $6$ singular fibers. It is well known that $\mathcal{M}_2$ is the
hyperelliptic mapping class group, consequently $n + 2s \equiv 0
\pmod{10}$. Therefore the only possible case for $k=6$ is $(n, s) =
(2,4)$. But Proposition~\ref{Proposition:SigmaEstimate} excludes this
case because there is no $\sigma(X) \in \mathbb{Z}$ such that
\[
-\frac{12}{3} \le \sigma(X) \le n -s -1 = 2 - 4- 1 = -3,
\]
and $4 | (6 + \sigma(X))$. 
\end{proof}

\bibliographystyle{amsplain}

\providecommand{\bysame}{\leavevmode\hbox to3em{\hrulefill}\thinspace}
\providecommand{\MR}{\relax\ifhmode\unskip\space\fi MR }
\providecommand{\MRhref}[2]{%
  \href{http://www.ams.org/mathscinet-getitem?mr=#1}{#2}
}
\providecommand{\href}[2]{#2}

\end{document}